\documentclass[12pt, reqno]{amsart}
\usepackage{amsmath, amsthm, amscd, amsfonts, amssymb, graphicx, color, mathrsfs}
\usepackage[bookmarksnumbered, colorlinks, plainpages]{hyperref}
\usepackage[all]{xy} 
\usepackage{slashed}

\newtheorem{theorem}{Theorem}[section]
\newtheorem{lemma}[theorem]{Lemma}

\theoremstyle{definition}
\newtheorem{definition}[theorem]{Definition}

\newtheorem{Open Problem}[theorem]{Open Problem}

\newtheorem{conjecture}[theorem]{Conjecture}

\theoremstyle{remark}
\newtheorem{remark}[theorem]{Remark}
\numberwithin{equation}{section}

\begin{document}
\setcounter{page}{1}

\title[ Oscillating  singular integral operators on graded groups revisited]{  Oscillating  singular integral operators on graded Lie groups revisited   }

\author[D. Cardona]{Duv\'an Cardona}
\address{
  Duv\'an Cardona:
  \endgraf
  Department of Mathematics: Analysis, Logic and Discrete Mathematics
  \endgraf
  Ghent University, Belgium
  \endgraf
  {\it E-mail address} {\rm duvanc306@gmail.com, duvan.cardonasanchez@ugent.be}
  }
  
\author[M. Ruzhansky]{Michael Ruzhansky}
\address{
  Michael Ruzhansky:
  \endgraf
  Department of Mathematics: Analysis, Logic and Discrete Mathematics
  \endgraf
  Ghent University, Belgium
  \endgraf
 and
  \endgraf
  School of Mathematical Sciences
  \endgraf
  Queen Mary University of London
  \endgraf
  United Kingdom
  \endgraf
  {\it E-mail address} {\rm michael.ruzhansky@ugent.be, m.ruzhansky@qmul.ac.uk}
  }

\thanks{The authors were supported  by the FWO  Odysseus  1  grant  G.0H94.18N:  Analysis  and  Partial Differential Equations and  by the Methusalem programme of the Ghent University Special Research Fund (BOF)
(Grant number 01M01021). Duv\'an Cardona was supported by the Research Foundation-Flanders
(FWO) under the postdoctoral
grant No 1204824N.  Michael Ruzhansky is also supported  by EPSRC grant 
EP/R003025/2.
}


\begin{abstract} In this work, we extend the Euclidean theory of oscillating singular integrals due to Fefferman and Stein in \cite{Fefferman1970,FeffermanStein1972} to arbitrary graded Lie groups. Our approach reveals the strong compatibility between the geometric measure theory of a graded Lie group and the Fourier analysis associated with Rockland operators. Our criteria are presented in terms of the oscillating Fefferman condition of the kernel of the operator and its group Fourier transform. One of the novelties of this work is that we use the infinitesimal representation of a Rockland operator to measure the decay of the Fourier transform of the kernel. 
\end{abstract} 

\maketitle

\tableofcontents
\allowdisplaybreaks

\section{Introduction}
\subsection{Outline}
This work is part of a series of papers \cite{MathZ1,MathZ2,CR20221} where the boundedness of oscillating integral operators on graded Lie groups and on compact Lie groups has been investigated. In \cite{CR20221}, we proved the $H^1$-$L^1$-boundedness of oscillating singular integrals on graded Lie groups, extending the analysis by Alexopoulos \cite{alexo} for the case of spectral multipliers of H\"ormander sub-Laplacians. For instance, in \cite{CR20221}, our criteria are presented according to the modern theory of PDE in terms of Rockland operators. Having proved the $H^1$-$L^1$ boundedness estimate in \cite{CR20221}, here we consider the problem of extending the weak $(1,1)$ estimate by Fefferman \cite{Fefferman1970} to graded Lie groups. For the regularity properties of more general classes of spectral multipliers (of non-oscillating type) and of pseudo-differential operators on graded Lie groups, we refer the reader to \cite{CardonaDelgadoRuzhansky}.

The results in \cite{CR20221}, together with the main result Theorem \ref{main:th} of this manuscript, extend the theory of oscillating singular integrals developed by Fefferman and Stein in \cite{Fefferman1970,FeffermanStein1972} to the setting of graded Lie groups. Our setting then covers a variety of groups of interest in analysis of PDE, namely, the Euclidean space, Heisenberg-type groups, stratified groups, etc. For our criteria, we will use the Fourier analysis associated with Rockland operators. These are hypoelliptic partial differential operators in view of the Helffer and Nourrigat solution of the Rockland conjecture (see \cite{HelfferNourrigat}), and they could be of arbitrary order. In the particular case of stratified Lie groups, they encompass the family of H\"ormander sub-Laplacians and their integer powers.

\subsection{Historical aspects}
We now turn to the historical aspects of the subject in the Euclidean setting that inspired our approach. As pointed out by Stein in \cite{Stein1998Notices}, there is probably no work in the last seventy years that has had as widespread an influence in analysis as the historic memoir \cite{CalderonZygmund1952}, {\it ``On the existence of certain singular integrals''} by Calder\'on and Zygmund, published in 1952 in {\it Acta Mathematica}. Using the methods of real interpolation (e.g., the Marcinkiewicz interpolation theorem), one of the main problems of the Calder\'on-Zygmund theory is to identify the sharp conditions on a distribution $K$ such that the convolution operator $f\mapsto f\ast K$ can be extended to a weak $(1,1)$-type operator, that is,
$$ \Vert f\ast K\Vert_{L^{1,\infty}(\mathbb{R}^n)}\leq C\Vert f\Vert_{L^1(\mathbb{R}^n)},\quad f\in C^\infty_0(\mathbb{R}^n). $$

For instance, singular integrals were important predecessors of so-called pseudo-differential operators, whose impact in analysis, partial differential equations, and differential geometry was recognized through the works of Kohn and Nirenberg, H\"ormander, Stein and Fefferman, and Atiyah and Singer, among others (see H\"ormander \cite[Page 178]{Hormander1985III} for historical details). In terms of convolution operators, a remarkable kernel class $H_{\infty}$ was introduced by H\"ormander, consisting of all distributions satisfying the kernel estimate
\begin{equation}\label{FeffCondINFTY}
[K]_{H_{\infty}} := \sup_{R>0}\left\Vert \smallint\limits_{|x|\geq 2R}|K(x-y)-K(x)|\,dx \right\Vert_{L^\infty(B(0,R),\,dy)} < \infty.
\end{equation}
Then, in \cite{Hormander1960}, H\"ormander proved that a convolution operator $Tf = K\ast f$ with a kernel satisfying the ``smoothness'' condition in \eqref{FeffCondINFTY} and bounded on $L^2(\mathbb{R}^n)$ (that is, the Fourier transform of the kernel $\widehat{K}\in L^\infty(\mathbb{R}^n)$) is of weak $(1,1)$-type.

The theory of oscillating singular integrals developed by Fefferman and Stein in \cite{Fefferman1970,FeffermanStein1972} extends in some aspects the work of Calder\'on and Zygmund by enabling the analysis of oscillating multipliers. Indeed, generalizing the H\"ormander condition, Fefferman \cite{Fefferman1970} and Fefferman and Stein \cite{FeffermanStein1972} considered distributions with compact support satisfying the condition
\begin{equation}\label{FeffCond}
[K]_{H_{\infty},\theta} := \sup_{0<R<1}\left\Vert \smallint\limits_{|x|\geq 2R^{1-\theta}}|K(x-y)-K(x)|\,dx \right\Vert_{L^\infty(B(0,R),\,dy)} < \infty.
\end{equation}

Roughly speaking, if $K$ satisfies \eqref{FeffCond} and its Fourier transform has the decay property
\begin{equation}\label{decay}
|\widehat{K}(\xi)| = O((1+|\xi|)^{-\frac{n\theta}{2}}),\quad 0 < \theta < 1,
\end{equation}
then Fefferman's theorem establishes that $T$ admits a bounded extension of weak $(1,1)$-type, provided the support of $K$ is sufficiently small. While the small support condition does not explicitly appear in \cite[Theorem 2']{Fefferman1970}, it is implicitly assumed in the proof given in \cite[Page 24]{Fefferman1970}. This assumption is not restrictive, as an analysis similar to that in \cite[Page 23]{Fefferman1970} enables the reduction from distributions $K$ with arbitrary support size to those with small support.   

Clearly, when $\theta = 0$, Fefferman's condition reduces to H\"ormander's condition for distributions $K$ with compact support. The boundedness of convolution operators $T$ with kernels satisfying the H\"ormander condition from the Hardy space $H^1(\mathbb{R}^n)$ to itself was established in \cite{FeffermanStein1972}.

From the perspective of harmonic analysis, our main Theorem \ref{main:th} provides a critical estimate for oscillating singular integrals on $L^p$-spaces at the endpoint $p=1$, specifically establishing their weak $(1,1)$-boundedness. This result reveals profound connections between Fourier analysis on graded Lie groups and their geometric measure theory through the lens of representation theory. A key technical tool enabling this work is the existence of hypoelliptic left-invariant linear partial differential operators on (graded) Lie groups, which follows from Helffer and Nourrigat's solution to the Rockland conjecture \cite{HelfferNourrigat}. 

A fundamental tool in proving our critical estimate is, as anticipated, the Calder\'on-Zygmund decomposition for spaces of homogeneous type developed by Coifman and Weiss in \cite{CoifmanWeiss}. The key innovation of this work lies in combining three essential components:
\begin{itemize}
    \item The Calder\'on-Zygmund decomposition of an integrable function $f = g + b = g + \sum_j b_j$ into its good part $g$ and bad parts $b_j$,
    \item The geometric properties of the supports of $b_j,$  and,
    \item the Fourier analysis associated with Rockland operators.
\end{itemize}
This synthesis is formally established in the proof of Lemma \ref{lemma:Fefferman}.

\subsection{The main result}
Let $T: X\rightarrow Y$ be a bounded operator between Banach spaces $X$ and $Y$, with operator norm denoted by $\|T\|_{\textnormal{op}}$. To present our main result, we first introduce the necessary notation:
\begin{itemize}
    \item $\widehat{G}$ denotes the unitary dual of $G$,
    \item $\mathcal{R} = \sum_{[\alpha]\leq \nu}a_\alpha X^\alpha$ is a positive Rockland operator on $G$, defined as a positive, left-invariant, hypoelliptic partial differential operator,
    \item $\pi(\mathcal{R}) = \widehat{k}_{\mathcal{R}}(\pi)$ represents the Fourier transform of its right-convolution kernel ${k}_{\mathcal{R}}$, characterized by the identity $\mathcal{R}f = f\ast {k}_{\mathcal{R}}$ for $f \in C^{\infty}_0(G)$.
\end{itemize}
The main theorem of this work is stated as follows:
\begin{theorem}\label{main:th}Let $G$ be a graded Lie group of homogeneous dimension $Q$, and let $T \colon C^\infty_0(G) \rightarrow \mathscr{D}'(G)$ be a left-invariant operator with right-convolution kernel $K \in L^1_{\textnormal{loc}}(G \setminus \{e\})$, that is $T$ is defined by $Tf = f \ast K$. Assume that:
\begin{enumerate}
    \item $K$ is a distribution with compact support and sufficiently small diameter, namely that, $\textnormal{diam}(\textnormal{supp}(K)) < 1.$
    \item For some $0 \leq \theta < 1$, $K$ satisfies the group Fourier transform conditions:
    \begin{equation}\label{LxiG}
        \sup_{\pi \in \widehat{G}} \| (1+\pi(\mathcal{R}))^{\frac{Q\theta}{2\nu}} \widehat{K}(\pi) \|_{\textnormal{op}} < \infty, \quad 
        \sup_{\pi \in \widehat{G}} \| \widehat{K}(\pi) (1+\pi(\mathcal{R}))^{\frac{Q\theta}{2\nu}} \|_{\textnormal{op}} < \infty;
    \end{equation}
    
    \item $K$ satisfies the oscillating H\"ormander condition:
    \begin{equation}\label{GS:CZ:cond}
        [K]_{H_{\infty,\theta}(G)} := \sup_{0<R<1} \sup_{|y|<R} \smallint\limits_{|x| \geq 2R^{1-\theta}} |K(y^{-1}x) - K(x)| \, dx < \infty.
    \end{equation}
\end{enumerate}
Then $T$ admits a bounded extension of weak $(1,1)$-type.
\end{theorem}

Now, we briefly discuss our results.

\begin{remark}
For $G = \mathbb{R}^n$, we recover the classical weak $(1,1)$ boundedness result proved by Fefferman in \cite[Theorem 2']{Fefferman1970}. From the proof of Theorem \ref{main:th} in Section \ref{proof:section}, one can estimate the operator norm of $T: L^1 \rightarrow L^{1,\infty}$ by
\[
\Vert T \Vert_{L^1 \rightarrow L^{1,\infty}} \lesssim \Vert T \Vert_{L^2 \rightarrow L^2} + [K]_{H_{\infty,\theta}} = \Vert \widehat{K} \Vert_{L^\infty} + [K]_{H_{\infty,\theta}},
\]
in view of the Plancherel theorem.
\end{remark}

\begin{remark}
The main contribution of Theorem \ref{main:th} is when $0 < \theta < 1$. For $\theta = 0$, the statement in Theorem \ref{main:th} follows from the general theory of Calder\'on-Zygmund operators developed by Coifman and Weiss in \cite{CoifmanWeiss}. H\"ormander-Mihlin criteria on graded groups were obtained by the second author and V. Fischer in \cite{FR}. For different classes of oscillating multipliers, we refer the reader to Bramati, Ciatti, Green, and Wright \cite{BramatietAl}, and to Ciatti and Wright \cite{CiattiWright2018} for the case of stratified Lie groups.
\end{remark}

\begin{remark}
The boundedness of an operator $T$ satisfying the hypotheses in Theorem \ref{main:th} from the Hardy space $H^1(G)$ into $L^1(G)$ has been proved by the authors in \cite{CR20221}, while the $H^1$-$L^1$ estimate for pseudo-differential operators on graded Lie groups was analyzed in \cite{CardonaDelgadoRuzhansky}. As for spectral multipliers, we refer the reader to the classical paper of Alexopoulos \cite{alexo}.
\end{remark}

\begin{remark}[Historical note]
Oscillating singular integrals arose as generalisations of oscillating Fourier multipliers. These are multipliers associated to symbols of the form
\begin{equation*}
    \widehat{K}(\xi) = \psi(\xi)\frac{e^{i|\xi|^a}}{|\xi|^{\frac{n\alpha}{2}}}, \quad \psi \in C^{\infty}(\mathbb{R}^n), \quad 0 < a < 1,
\end{equation*}
where $\psi$ vanishes near the origin and is equal to one for large $|\xi|$. It was proved by Wainger \cite{Wainger1965} that $K(x)$ is essentially equal to $c_n |x|^{-n-\lambda} e^{ic_n' |x|^{a'}}$, where
$\lambda = \frac{n(a-\alpha)}{2(1-a)}$ and $a' = \frac{a}{a-1}$. From this, one can deduce that $|\nabla K(x)| \lesssim |x|^{-n-\lambda-1+a'}$, from which it follows that for $a = \alpha = \theta$, the estimate \eqref{GS:CZ:cond} remains valid.

The $L^p$ properties for convolution operators of this form were first studied by Hardy \cite{Hardy1913}, Hirschman \cite{Hirschman1956}, and Wainger \cite{Wainger1965}. The sharp version of the $L^\infty$-BMO boundedness for oscillating Fourier multipliers can be deduced from the classical work of Fefferman \cite{Fefferman1973}. Further works on the subject in the setting of manifolds and beyond can be found in Seeger \cite{Seeger,Seeger1990,Seeger1991}, Seeger and Sogge \cite{SeegerSogge}, and for the setting of Fourier integral operators, we refer the reader to Seeger, Sogge, and Stein \cite{SSS} and Tao \cite{Tao}.
\end{remark}

\subsection{Open problems} In addition to our main result, in this paper we propose Conjecture \ref{Conjecture} and Open Problem \ref{OpenPro}, both related to a modern problem in harmonic analysis: namely, the determination of optimal ranges of boundedness results in terms of the topological dimension $n$ of the group, rather than its Hausdorff dimension. This concerns, in particular, the number or order of derivatives required of the symbol. For the relevant notions of graded Lie groups and Rockland operators, we refer the reader to Section \ref{preliminaries}.

\begin{conjecture}\label{Conjecture}
    Let $G$ be a graded Lie group of homogeneous dimension $Q$, and let $\mathcal{R}$ be a positive Rockland operator of homogeneous degree $\nu>0$. Then, the weak $(1,1)$ boundedness of the oscillating spectral multiplier
    \[
    e^{i (1+\mathcal{R})^{a/\nu}}(1+\mathcal{R})^{-m/\nu}, \quad m = \frac{Qa}{2}, \quad 0 < a < 1,
    \]
    can be deduced from the boundedness result in Theorem \ref{main:th}. Moreover, its $H^1$–$L^1$ boundedness can be shown by proving that it belongs to the H\"ormander class $S^{-m}_{1-a,0}(G \times \widehat{G})$ on the phase space $G \times \widehat{G}$, where the boundedness properties for these classes have been established in \cite{CardonaDelgadoRuzhansky}.
\end{conjecture}

\begin{Open Problem}\label{OpenPro}
The problem of improving the index $Q$—corresponding to the Hausdorff dimension of the group—in the weak $(1,1)$ estimate \eqref{LxiG} of Theorem \ref{main:th} by replacing it with the topological dimension $n$, or with another number $N > 0$ such that $n \leq N < Q$, remains open. Specifically, one may ask whether  conditions of the form
\begin{equation}\label{LxiG:2:Open}
  \sup_{\pi \in \widehat{G}} \left\| (1+\pi(\mathcal{R}))^{\frac{N\theta}{2\nu}} \widehat{K}(\pi) \right\|_{\textnormal{op}} < \infty,\quad \sup_{\pi \in \widehat{G}} \left\| \widehat{K}(\pi) (1+\pi(\mathcal{R}))^{\frac{N\theta}{2\nu}} \right\|_{\textnormal{op}} < \infty,
\end{equation}
together with the oscillating H\"ormander condition in \eqref{GS:CZ:cond}, suffices to guarantee the weak $(1,1)$ boundedness (or the $H^1$–$L^1$ boundedness) of the operator $T$.

This problem remains open even in the setting where $G$ is a general stratified Lie group and $\mathcal{R}$ is an arbitrary Rockland operator of any order. For the current state of the art when $\mathcal{R}$ is a second-order operator given by a H\"ormander sub-Laplacian $\mathcal{L} = -\sum_i X_i^2$, we refer the reader to the pioneering work of Alexopoulos \cite{alexo}, as well as to the more recent contributions by Martini, M\"uller, and Nicolussi Golo \cite{MartiniMullerNicolussi}, and the extensive references therein.
\end{Open Problem}

\section{Fourier analysis on graded groups}\label{preliminaries}

The notation and terminology of this paper on the analysis of homogeneous Lie groups are mainly taken from Folland and Stein \cite{FollandStein1982}. For the analysis of Rockland operators, we will follow \cite[Chapter 4]{FischerRuzhanskyBook}.

\subsection{Homogeneous and graded Lie groups} 
    Let $G$ be a homogeneous Lie group. This means that $G$ is a connected and simply connected Lie group whose Lie algebra $\mathfrak{g}$ is endowed with a family of dilations. We introduce it in the following definition.

\begin{definition}
A family of dilations $D_{r}^{\mathfrak{g}}$, $r > 0$, on the Lie algebra $\mathfrak{g}$ is a family of automorphisms on $\mathfrak{g}$ satisfying the following two conditions:
\begin{itemize}
\item For every $r > 0$, $D_{r}^{\mathfrak{g}}$ is a map of the form
$$ D_{r}^{\mathfrak{g}} = \textnormal{Exp}(\ln(r)A) $$
for some diagonalisable linear operator $A \equiv \textnormal{diag}[\nu_1, \cdots, \nu_n]$ on $\mathfrak{g}$.
\item For all $X, Y \in \mathfrak{g}$ and $r > 0$, 
$$ [D_{r}^{\mathfrak{g}}X, D_{r}^{\mathfrak{g}}Y] = D_{r}^{\mathfrak{g}}[X, Y]. $$ 
\end{itemize}
\end{definition}
\begin{remark}
We call the eigenvalues of $A$, $\nu_1, \nu_2, \dots, \nu_n$, the dilation weights or weights of $G$. 
\end{remark}

In our analysis, the notion of the homogeneous dimension of the group is crucial. We introduce it as follows.

\begin{definition}
The homogeneous dimension of a homogeneous Lie group $G$ is given by
$$ Q = \textnormal{\textbf{Tr}}(A) = \nu_1 + \cdots + \nu_n.  $$
\end{definition}

\begin{definition}[Dilations on the group]
The dilations $D_{r}^{\mathfrak{g}}$ of the Lie algebra $\mathfrak{g}$ induce a family of maps on $G$ defined via
$$ D_{r} := \exp_{G} \circ D_{r}^{\mathfrak{g}} \circ \exp_{G}^{-1}, \quad r > 0, $$
where $\exp_{G}: \mathfrak{g} \rightarrow G$ is the usual exponential map associated with the Lie group $G$. We refer to the family $D_{r}$, $r > 0$, as dilations on the group.
\end{definition}
 \begin{remark}
If we write $rx = D_{r}(x)$, $x \in G$, $r > 0$, then a relation on the homogeneous structure of $G$ and the Haar measure $dx$ on $G$ is given by
$$ \smallint\limits_{G}(f \circ D_{r})(x)\,dx = r^{-Q} \smallint\limits_{G} f(x)\,dx. $$
\end{remark}

\begin{remark}
A Lie group is graded if its Lie algebra $\mathfrak{g}$ may be decomposed as the sum of subspaces 
$$ \mathfrak{g} = \mathfrak{g}_{1} \oplus \mathfrak{g}_{2} \oplus \cdots \oplus \mathfrak{g}_{s}, $$ 
such that $[\mathfrak{g}_{i}, \mathfrak{g}_{j}] \subset \mathfrak{g}_{i+j}$, and $\mathfrak{g}_{i+j} = \{0\}$ if $i + j > s$.
\end{remark}

Examples of graded Lie groups are the Heisenberg group $\mathbb{H}^n$ and, more generally, any stratified groups where the Lie algebra $\mathfrak{g}$ is generated by $\mathfrak{g}_{1}$. Here, $n$ is the topological dimension of $G$, $n = n_{1} + \cdots + n_{s}$, where $n_{k} = \dim \mathfrak{g}_{k}$.

\begin{remark}
A Lie algebra admitting a family of dilations is nilpotent, and hence so is its associated
connected, simply connected Lie group. The converse does not hold, i.e., not every
nilpotent Lie group is homogeneous, although they exhaust a large class. See \cite{FischerRuzhanskyBook} for details. Indeed, the main class of Lie groups under our consideration is that of graded Lie groups.
\end{remark}

\begin{remark}
A graded Lie group $G$ is a homogeneous Lie group equipped with a family of weights $\nu_j$, all of them positive rational numbers. Let us observe that if $\nu_{i} = \frac{a_i}{b_i}$ with $a_i, b_i$ integers, and $b$ is the least common multiple of the $b_i$'s, the family of dilations 
$$ \mathbb{D}_{r}^{\mathfrak{g}} = \textnormal{Exp}(\ln(r^b)A): \mathfrak{g} \rightarrow \mathfrak{g}, $$
has integer weights, $\nu_{i} = \frac{a_i b}{b_i}$. Thus, in this paper, we always assume that the weights $\nu_j$, defining the family of dilations, are non-negative integers, which allows us to assume that the homogeneous dimension $Q$ is a non-negative integer. This is a natural context for the study of Rockland operators (see Remark 4.1.4 of \cite{FischerRuzhanskyBook}).
\end{remark}

\subsection{Fourier analysis on nilpotent Lie groups}

Let $G$ be a simply connected nilpotent Lie group. Then, the adjoint representation $\textnormal{ad}:\mathfrak{g}\rightarrow\textnormal{End}(\mathfrak{g})$ is nilpotent. Next, we define unitary and irreducible representations.

\begin{definition}
We say that $\pi$ is a continuous, unitary, and irreducible representation of $G$ if the following properties are satisfied:
\begin{itemize}
    \item $\pi\in \textnormal{Hom}(G, \textnormal{U}(H_{\pi}))$ for some separable Hilbert space $H_\pi$; that is, $\pi(xy)=\pi(x)\pi(y)$ and, for the adjoint of $\pi(x)$, we have $\pi(x)^*=\pi(x^{-1})$ for every $x,y\in G$.
    \item The map $(x,v)\mapsto \pi(x)v$ from $G\times H_\pi$ into $H_\pi$ is continuous.
    \item For every $x\in G$ and for every subspace $W_\pi\subset H_\pi$, if $\pi(x)W_{\pi}\subset W_{\pi}$, then $W_\pi=H_\pi$ or $W_\pi=\{0\}$.
\end{itemize}
\end{definition}

\begin{definition}[The unitary dual]
Let $\textnormal{Rep}(G)$ be the set of unitary, continuous, and irreducible representations of $G$. The relation
{\small{
\begin{equation*}
    \pi_1\sim \pi_2\textnormal{ if and only if there exists } A\in \mathscr{B}(H_{\pi_1},H_{\pi_2}) \textnormal{ such that } A\pi_{1}(x)A^{-1}=\pi_2(x),
\end{equation*}}}
for every $x\in G$, is an equivalence relation, and the unitary dual of $G$, denoted by $\widehat{G}$, is defined via
\[
    \widehat{G}:={\textnormal{Rep}(G)}/{\sim}.
\]
Let us denote by $d\pi$ the Plancherel measure on $\widehat{G}$.
\end{definition}

Next, we define the main object for our further analysis: the Fourier transform.

\begin{definition}[Group Fourier Transform]
The Fourier transform of $f\in \mathscr{S}(G)$ (meaning that $f\circ \textnormal{exp}_G\in \mathscr{S}(\mathfrak{g})$, with $\mathfrak{g}\simeq \mathbb{R}^{\dim(G)}$) at $\pi\in\widehat{G}$ is defined by
\begin{equation*}
    \widehat{f}(\pi)=\smallint\limits_{G}f(x)\pi(x)^*\,dx:H_\pi\rightarrow H_\pi, \quad \textnormal{and} \quad \mathscr{F}_{G}:\mathscr{S}(G)\rightarrow \mathscr{S}(\widehat{G}):=\mathscr{F}_{G}(\mathscr{S}(G)).
\end{equation*}
\end{definition}

\begin{remark}[Fourier Inversion Formula and Plancherel Theorem]
Identifying each representation \(\pi\) with its equivalence class \([\pi]=\{\pi' : \pi\sim \pi'\}\), for every \(\pi\in \widehat{G},\) the Kirillov trace character \(\Theta_\pi\) is defined by
\[
(\Theta_{\pi},f) := \textnormal{\textbf{Tr}}(\widehat{f}(\pi)),
\]
where \(\textnormal{\textbf{Tr}}\) denotes the trace of trace class operators on $H_\pi$. This defines \(\Theta_\pi\) as a tempered distribution on \(\mathscr{S}(G)\). 
In particular, the identity
\[
    f(e_G)=\smallint\limits_{\widehat{G}}(\Theta_{\pi},f)\,d\pi
\]
implies the Fourier inversion formula \(f=\mathscr{F}_G^{-1}(\widehat{f})\), where
\[
    (\mathscr{F}_G^{-1}\sigma)(x):=\smallint\limits_{\widehat{G}}\textnormal{\textbf{Tr}}(\pi(x)\sigma(\pi))\,d\pi, \quad x\in G,
\]
defines the inverse Fourier transform \(\mathscr{F}_G^{-1}:\mathscr{S}(\widehat{G})\to\mathscr{S}(G)\). In this context, the Plancherel theorem asserts that
\[
    \|f\|_{L^2(G)} = \|\widehat{f}\|_{L^2(\widehat{G})},
\]
where \(L^2(\widehat{G})\) is the Hilbert space defined by
\[
    L^2(\widehat{G}) := \smallint\limits_{\widehat{G}} H_\pi\otimes H_{\pi}^*\,d\pi,
\]
equipped with the norm
\[
    \|\sigma\|_{L^2(\widehat{G})} := \left( \smallint\limits_{\widehat{G}} \|\sigma(\pi)\|_{\textnormal{HS}}^2\,d\pi \right)^{\frac{1}{2}},
\]
and \(\|\sigma(\pi)\|_{\textnormal{HS}}\) denotes the Hilbert--Schmidt norm of \(\sigma(\pi)\).
\end{remark}

\subsection{Homogeneous linear operators and Rockland operators} There is a family of continuous linear operators that respect the action of the dilations of the group. These are called \emph{homogeneous linear operators}. We introduce them in the following definition.
\begin{definition}[Homogeneous operators]
A linear operator \(T:C^\infty(G)\rightarrow \mathscr{D}'(G)\) is said to be \emph{homogeneous of degree} \(\nu\in \mathbb{C}\) if, for every \(r>0\), the equality
\[
T(f\circ D_{r}) = r^{\nu}(Tf)\circ D_{r}
\]
holds for every \(f\in \mathscr{D}(G)\).
\end{definition}
Now, we introduce the main class of differential operators in the context of nilpotent Lie groups. The existence of these operators characterizes the family of graded Lie groups. We call them \emph{Rockland operators}.

\begin{definition}[Rockland operators]
Let \(\pi \in \widehat{G}\) be a unitary irreducible representation \(\pi:G\to U(H_\pi)\). We denote by \(H_{\pi}^{\infty}\) the space of smooth vectors, that is, the set of all \(v\in H_\pi\) such that the map \(x\mapsto \pi(x)v\), \(x\in G\), is smooth.

A \emph{Rockland operator} is a left-invariant differential operator \(\mathcal{R}\) which is homogeneous of positive degree \(\nu=\nu_{\mathcal{R}}>0\), and such that for every non-trivial unitary irreducible representation \(\pi\in \widehat{G}\), the operator \(\pi(\mathcal{R})\) is injective on \(H_{\pi}^{\infty}\). 

The map \(\sigma_{\mathcal{R}}(\pi)=\pi(\mathcal{R})\) is called the \emph{symbol} associated to \(\mathcal{R}\); it coincides with the infinitesimal representation of \(\mathcal{R}\) viewed as an element of the universal enveloping algebra.
\end{definition}
\begin{remark}
It can be shown that a Lie group \(G\) is graded if and only if there exists a differential Rockland operator on \(G\).
\end{remark}

Next, we record for our further analysis some aspects of the functional calculus for Rockland operators.

\begin{remark}[Functional calculus for Rockland operators]
If the Rockland operator \(\mathcal{R}\) is formally self-adjoint, then both \(\mathcal{R}\) and \(\pi(\mathcal{R})\) admit self-adjoint extensions on \(L^{2}(G)\) and \(H_{\pi}\), respectively. Preserving the same notation for their self-adjoint extensions, and denoting by \(E\) and \(E_{\pi}\) their corresponding spectral measures, we define the functional calculus by
\[
f(\mathcal{R}) = \smallint_{-\infty}^{\infty} f(\lambda) \, dE(\lambda), \quad \text{and} \quad \pi(f(\mathcal{R})) \equiv f(\pi(\mathcal{R})) = \smallint_{-\infty}^{\infty} f(\lambda) \, dE_{\pi}(\lambda).
\]
In general, we will reserve the notation \(\{dE_A(\lambda)\}_{0\leq \lambda <\infty}\) for the spectral measure associated with a positive and self-adjoint operator \(A\) on a Hilbert space \(H\).
\end{remark}
We now recall a lemma on dilations on the unitary dual \(\widehat{G}\), which will be useful in our analysis of spectral multipliers. For the proof, see Lemma 4.3 of \cite{FischerRuzhanskyBook}.

\begin{lemma}\label{dilationsrepre}
For every \(\pi\in \widehat{G}\) let us define
\begin{equation}\label{dilations:repre}
    D_{r}(\pi)(x) \equiv (r\cdot \pi)(x) := \pi(r\cdot x) \equiv \pi(D_r(x)),
\end{equation}
for every \(r>0\) and all \(x\in G\). Then, if \(f\in L^{\infty}(\mathbb{R})\), we have
\[
f(\pi^{(r)}(\mathcal{R})) = f(r^{\nu} \pi(\mathcal{R})).
\]
\end{lemma}
\begin{remark}
For instance, for any \(\alpha\in \mathbb{N}_0^n\) and for an arbitrary family \(X_1, \dots, X_n\) of left-invariant vector fields, we will use the notation
\begin{equation}
    [\alpha] := \sum_{j=1}^{n} \nu_j \alpha_j,
\end{equation}
for the homogeneous degree of the operator \(X^{\alpha} := X_1^{\alpha_1} \cdots X_n^{\alpha_n}\), whose order is \(|\alpha| := \sum_{j=1}^{n} \alpha_j\).
\end{remark}
\begin{remark}\label{DilationLebesgue}
By considering the dilation \(r\cdot x = D_{r}(x)\), \(x\in G\), \(r>0\), the relation between the homogeneous structure of \(G\) and the Haar measure \(dx\) on \(G\) is given by (see \cite[Page 100]{FischerRuzhanskyBook})
\[
\smallint\limits_{G} (f\circ D_{r})(x) \, dx = r^{-Q} \smallint\limits_{G} f(x) \, dx.
\]
Note that if we define \(f_{r} := r^{-Q} f(r^{-1}\cdot)\), then
\begin{equation}\label{Eq:dilatedFourier}
    \widehat{f}_{r}(\pi) = \smallint\limits_{G} r^{-Q} f(r^{-1} \cdot x) \pi(x)^* \, dx = \smallint\limits_{G} f(y) \pi(r\cdot y)^{*} \, dy = \widehat{f}(r\cdot \pi),
\end{equation}
for any \(\pi\in \widehat{G}\) and all \(r>0\), with \((r\cdot \pi)(y) = \pi(r\cdot y)\), \(y\in G\), as in \eqref{dilations:repre}.
\end{remark}

The following lemma present the action of the dilations of the group $G$ into the kernels of bounded functions of a Rockland operator $\mathcal{R},$ see \cite[Page 179]{FischerRuzhanskyBook}.
\begin{lemma}\label{Fundamental:lemmaCZ:graded}
Let $\varkappa\in L^{\infty}(\mathbb{R}^{+}_0)$ and let $r>0.$ Then, we have the following identity
\begin{equation}
    \varkappa(r^{\nu}\mathcal{R})\delta(x)=r^{-Q}[\varkappa(\mathcal{R})\delta](r^{-1}\cdot x),
\end{equation}for all $x\in G.$
\end{lemma}

\section{Proof of the main theorem}\label{proof:section}

\subsection{Part 1: The Calder\'on-Zygmund decomposition}

Let $G$ be a graded Lie group with homogeneous dimension $Q$, and let $\mathcal{R}$ be a positive Rockland operator on $G$ of homogeneous degree $\nu>0$. In this section, we will prove that for a singular integral operator $T$ satisfying conditions \eqref{LxiG} and \eqref{GS:CZ:cond}, there exists a constant $C>0$ such that
\begin{equation}\label{To:proof} \Vert T f\Vert_{L^{1,\infty}(G)}:= \sup_{\alpha>0} \alpha|{x\in G:|Tf(x)|>\alpha}|\leq C\Vert f\Vert_{L^1(G)}, \end{equation} with $C$ independent of $f$.

We start by analysing the condition in \eqref{LxiG}.

\begin{remark}\label{Rem:Four:Conf}
Observe that, in view of the Plancherel theorem, the hypothesis \eqref{LxiG} in Theorem \ref{main:th} implies that the operator $(1+\mathcal{R})^{\frac{Q\theta}{2\nu}}T$ admits a bounded extension on $L^2(G)$. Indeed, the Plancherel theorem implies that
\begin{align*}
     \Vert (1+\mathcal{R})^{\frac{Q\theta}{2\nu}}Tf\Vert_{L^2(G)}^2 &= \smallint\limits_{ \widehat{G}} \|(1+\pi(\mathcal{R}))^{\frac{Q\theta}{2\nu}}\widehat{K}(\pi)\widehat{f}(\pi)\|_{\textnormal{HS}}^2 \, d\pi\\
     &\leq \smallint\limits_{ \widehat{G}} \|(1+\pi(\mathcal{R}))^{\frac{Q\theta}{2\nu}}\widehat{K}(\pi)\|^2_{\textnormal{op}} \|\widehat{f}(\pi)\|_{\textnormal{HS}}^2 \, d\pi,
\end{align*}
and in view of \eqref{LxiG} we obtain
$$ \Vert (1+\mathcal{R})^{\frac{Q\theta}{2\nu}}Tf\Vert_{L^2(G)}^2 \leq C^2 \smallint\limits_{ \widehat{G}} \|\widehat{f}(\pi)\|_{\textnormal{HS}}^2 \, d\pi = C^2 \Vert f\Vert_{L^2(G)}^2, $$
proving the boundedness of $(1+\mathcal{R})^{\frac{Q\theta}{2\nu}}T$ on $L^2(G).$ Since $(1+\mathcal{R})^{-\frac{Q\theta}{2\nu}}: L^2(G)\rightarrow L^2(G)$ is bounded, we deduce that
\[
T= (1+\mathcal{R})^{-\frac{Q\theta}{2\nu}}(1+\mathcal{R})^{\frac{Q\theta}{2\nu}}T
\]
is also bounded on $L^2(G).$ 

Note also that the condition
$$
\sup_{\pi\in \widehat{G}} \Vert \widehat{K}(\pi)(1+\pi(\mathcal{R}))^{\frac{Q\theta}{2\nu}}\Vert_{\textnormal{op}}<\infty,
$$
provides the boundedness of the operator $T(1+\mathcal{R})^{\frac{Q\theta}{2\nu}}$ from $L^2(G)$ to itself.
\end{remark}

For the proof of \eqref{To:proof}, fix $f\in L^1(G),$ and let us consider its Calder\'on-Zygmund decomposition, see Coifman and Weiss \cite[Pages 73-74]{CoifmanWeiss}. So, for any $\gamma,\alpha>0, $ 
we  have the decomposition  $$f=g+b=g+\sum_{j}b_j,$$ where the following properties are satisfied. 
\begin{itemize}
    \item[(1)] $\Vert g \Vert_{L^\infty}\lesssim_{G} \gamma \alpha$ and $\Vert g\Vert_{L^1}\lesssim_{G} \Vert f\Vert_{L^1}.$
    \item[(2)] The $b_j$'s are supported in  open balls $I_j=B(x_j,r_j)$ where they satisfy the cancellation property
    \begin{equation}
        \smallint\limits_{I_j}b_j(x)dx=0.
    \end{equation}
     \item[(3)] Any component $b_j$ satisfies the $L^1$-estimate 
     \begin{equation}
        \Vert b_j\Vert_{L^1}\lesssim_{G} (\gamma \alpha)|I_j|.
\end{equation}
    \item[(4)] The sequence $\{|I_{j}|\}_{j}\in \ell^1$ and
    \begin{equation}\label{I:j}
    \sum_{j} |I_j|\lesssim_{G} (\gamma\alpha)^{-1}    \Vert f\Vert_{L^1}.   
    \end{equation}
   \item[(5)] 
    $$\Vert b\Vert_{L^1}\leq\sum_j  \Vert b_j\Vert_{L^1}\lesssim_{G} \Vert f\Vert_{L^1}.$$
    \item[(6)] There exists $M_0\in \mathbb{N},$ such that any point $x\in G$ belongs at most to $M_0$ balls of the collection $I_j.$
\end{itemize}

So, by fixing $\alpha\gamma>0,$ note that in terms of $g,$ $b$ and of $f$ one has the trivial estimate 
\begin{align*}
    |\{x:|Tf(x)|>\alpha\}|\leq |\{x:|Tg(x)|>\alpha/2\}|+|\{x:|Tb(x)|>\alpha/2\}|. 
\end{align*}
The estimates $\Vert g \Vert_{L^\infty}\lesssim \gamma \alpha$ and $\Vert g\Vert_{L^1}\lesssim \Vert f\Vert_{L^1},$ imply that
\begin{align*}
     \Vert g\Vert_{L^2}^2\leq\Vert g\Vert_{L^\infty}\Vert g\Vert_{L^1}\lesssim  (\gamma \alpha)\Vert f\Vert_{L^1}. 
\end{align*}
So, by applying the Chebishev inequality and the $L^2$-boundedness of $T,$ we have 
\begin{align*}
   &|\{x:|Tg(x)|>\alpha/2\}|\lesssim 2^2\alpha^{-2}\Vert Tg \Vert^2_{L^2}\leq (2\Vert T \Vert_{\mathscr{B}(L^2)})^2\alpha^{-2}\Vert g\Vert_{L^2}^2 \\
  &\leq (2\Vert T \Vert_{\mathscr{B}(L^2)})^2\alpha^{-2}(\gamma \alpha)\Vert f\Vert_{L^1}\lesssim\Vert T \Vert_{\mathscr{B}(L^2)}^2\gamma\alpha^{-1}\Vert f\Vert_{L^1}\\
   &\lesssim_{\gamma} \alpha^{-1}\Vert f\Vert_{L^1}.
\end{align*}
 In what  follows,  let us denote   $I^*=\bigcup I_j^*,$ where $$ I_j^*=B(x_j,2r_j)=\{x\in G:|x_{j}^{-1}x|<2R_j\},$$  and let us make use of the doubling condition in order to have the estimate
$$  |I^*_j|\sim  2^Q|I_j|, $$ from which it follows that 
$$ |I^*|\lesssim \sum_j|I_j^{*}| \lesssim \sum_j|I_j|\lesssim  \gamma^{-1}\alpha^{-1}\Vert f\Vert_{L^1}.   $$
Consequently, we have the estimates 
$$ |\{x:|Tb(x)|>\alpha/2\}|\leq |I^*|+ |\{x\in G\setminus I^* :|Tb(x)|>\alpha/2\}|  $$
$$ \lesssim\gamma^{-1}\alpha^{-1}\Vert f\Vert_{L^1}+|\{x\in G\setminus I^* :|Tb(x)|>\alpha/2\}| .  $$
$$ \lesssim_{\gamma} \alpha^{-1}\Vert f\Vert_{L^1}+|\{x\in G\setminus I^* :|Tb(x)|>\alpha/2\}| .  $$
So, to conclude the inequality \eqref{To:proof} we have to prove that 
\begin{equation}\label{To:proof:2}
  \sup_{\alpha>0}  \alpha|\{x\in G\setminus I^*:|Tb(x)|>\alpha/2\}|\leq C\Vert f\Vert_{L^1},
\end{equation}with $C$ independent of $f.$ So,  the proof of  Theorem \ref{main:th} consists of estimating the term $$|\{x\in G\setminus I^*:|Tb(x)|>\alpha/2\}|.$$

\subsection{Part 2: Proof of the oscillating case} The relevant case in Theorem \ref{main:th} is when $\theta \in (0,1)$. We begin the proof of the main theorem by explaining this claim.

\begin{proof}[Proof of Theorem \ref{main:th}]
We first consider the case $0 < \theta < 1$. Indeed, the statement for $\theta = 0$ in Theorem \ref{main:th} follows from the fundamental theorem of singular integrals due to Coifman and Weiss; see \cite[Theorem 2.4, Page 74]{CoifmanWeiss}. In fact, a graded Lie group satisfies the doubling condition, and hence it is a homogeneous topological space in the sense of Coifman and Weiss \cite{CoifmanWeiss}.

For $0 < \theta < 1$, we require some geometric transformations on the support of $K$, which we present in the next subsection. To achieve this, we will use two modern techniques: the stability of the powers of Rockland operators within the H\"ormander classes $S^{m}_{1,0}(G \times \widehat{G})$, and the Calder\'on-Vaillancourt theorem on graded Lie groups, as presented in \cite[Section 5.7, Theorem 5.7.1]{FischerRuzhanskyBook}.

Following our hypothesis, from now on, let us assume that the diameter of the support of $K$ is small, for instance, that
$$ \textnormal{diam}(\textnormal{supp}(K)) < c \leq 1, $$
where $0 < c < 1$ is suitable for our purposes. With this in mind, the next step is to construct a good replacement $\tilde{b}$ for $b$, which allows us to exploit the hypothesis in \eqref{GS:CZ:cond}.

\subsection{Part 3: The function $\phi.$ New family of supports}
Let us consider a function $\phi$ on $G$ such that
\begin{equation}
    \smallint\limits_{G}\phi(x)dx=1,\textnormal{  and  } \phi\in C^{\infty}_0(G,[0,\infty))\cap\textnormal{Dom}[\mathcal{R}^{-\frac{Q\theta}{2\nu}}].
\end{equation}
For $\varepsilon>0,$ 
define
\begin{equation}
    \phi(y,\varepsilon):=\varepsilon^{-Q}\phi(\varepsilon^{-1}\cdot x).
\end{equation}
Now, for any $j,$ define
\begin{equation}
    \phi_{j}(y):=\phi(y,2^{-\frac{1}{1-\theta}}\textnormal{diam}(I_j)^{\frac{1}{1-\theta}}),\,\,\,
\end{equation}
\begin{equation}
   \tilde{b}_{j}:=b_{j}(\cdot)*\phi_{j},\,\,
\end{equation}and 
\begin{equation}
    \tilde{b}:=\sum_{j}\tilde{b}_{j}.
\end{equation}Note that
\begin{equation}\label{bjk:re:def}
  Tb(x)=\sum_{j} T{b}_{j}(x),
\end{equation}for a.e. $x\in G.$
It is important to mention that in \eqref{bjk:re:def}, the sums on the right-hand side only run over $j$ such that $\textnormal{diam}(I_j) < c$. Indeed, for all $x \in G \setminus I^*$, the property of the support $\textnormal{diam}(\textnormal{supp}(K)) < c$ implies that for all $j$ with $\textnormal{diam}(I_j) \geq c$, we have
$$ T b_j = b_j \ast K = 0. $$
So, we only need to analyze the case where $\textnormal{diam}(I_j) < c$. Indeed, for $x \in G \setminus I^*$ and $j$ such that $\textnormal{diam}(I_j) \geq c$, we have
$$  b_{j} \ast K(x) = \int_{I_j} K(y^{-1}x) b_j(y) \, dy. $$
Since in the integral above $x \in G \setminus I^*$ and $y \in I_j$, we have
$$ |y^{-1}x| = \textnormal{dist}(x, y) > \textnormal{diam}(I_j) > c, $$
we conclude that the element $y^{-1}x$ is not in the support of $K$, and thus the integral vanishes.

Now, going back to the analysis of \eqref{To:proof:2},
note that
\begin{align*}
  &|\{x\in G\setminus I^*:|Tb(x)|>\frac{\alpha}{2}\}|\\
  &\leq |\{x\in G\setminus I^*:|Tb(x)-T\tilde{b}(x)|>\frac{\alpha}{4}\}|+|\{x\in G\setminus I^*:|T\tilde b(x)|>\frac{\alpha}{4}\}|\\
  &\leq \frac{4}{\alpha}\Vert T(b-\tilde{b})\Vert_{L^1(G\setminus I^*)}+|\{x\in G\setminus I^*:|T\tilde b(x)|>\frac{\alpha}{4}\}|,
\end{align*}
and let us take into account the estimate:
\begin{align*}
    \Vert T(b-\tilde{b})\Vert_{L^1(G\setminus I^*)} &=\smallint\limits_{ G\setminus I^*}|Tb(x)-T\tilde{b}(x)|dx\\
    &\leq \sum_{j}\smallint\limits_{ G\setminus I^*}|Tb_{j}(x)-T\tilde{b}_{j}(x)|dx.
\end{align*} We are going to prove that $T \tilde{b}$ and $T \tilde{b}_j$ are good replacements for $T b$ and $T b_j$, respectively, on the set $G \setminus I^*$.  
Observe that 
\begin{align*}
  & \smallint\limits_{ G\setminus I^*}|Tb_{j}(x)-T\tilde{b}_{j}(x)|dx=\smallint\limits_{ G\setminus I^*}| b_{j}\ast K(x)-\tilde{b}_{j}\ast K(x)|dx \\
   &=\smallint\limits_{ G\setminus I^*}| b_{j}\ast K(x)-[{b}_{j}\ast \phi_j\ast K](x)|dx \\
   &=\smallint\limits_{ G\setminus I^*}\left|\smallint\limits_{I_j} K(y^{-1}x)b_{j}(y)dy-\smallint\limits_{I_j}(\phi_j\ast K)(y^{-1}x){b}_{j}(y)dy\right|dx \\
   &\leq\smallint\limits_{I_j} \smallint\limits_{ G\setminus I^*}| K(y^{-1}x)-\phi_j\ast K(y^{-1}x)|dx |b_j(y)|dy \\
    &\leq\smallint\limits_{I_j} \smallint\limits_{ |z|>\textnormal{diam}(I_j)}| K(z)-\phi_j\ast K
    (z)|dz |b_{j}(y)|dy\\
    &\leq\smallint\limits_{I_j} \smallint\limits_{ |z|>\textnormal{diam}(I_j)}| K(z)-\phi_j\ast K(z)|dz |b_{j}(y)|dy\\
    &=\smallint\limits_{I_j}|b_{j}(y)|dy \smallint\limits_{ |z|>\textnormal{diam}(I_j)}| K(z)-\phi_j\ast K(z)|dz,
\end{align*}where, in the last line, we have used the change of variables $x \mapsto z = y^{-1}x$, and then we observe that $|z| > \textnormal{diam}(I_j)$ when $x \in G \setminus I^*$ and $y \in I_j.$ Using that $\phi_j$ is supported in a ball of radius 
$$ R_j := 2^{-\frac{1}{1-\theta}} \textnormal{diam}(I_j)^{\frac{1}{1-\theta}}, $$
and that $\|\phi_j\|_{L^1} = 1,$ we have that
\begin{align*}
   & \smallint\limits_{ |z|>\textnormal{diam}(I_j)  }| K(z)-\phi_j\ast K(z)|dz\\
   &=\smallint\limits_{ |z|>\textnormal{diam}(I_j)  }\left| K(z)\smallint\limits_{|y|<2^{-\frac{1}{1-\theta}}\textnormal{diam}(I_j)^{\frac{1}{1-\theta}}}\phi_{j}(y)dy-\smallint\limits_{|y|<2^{-\frac{1}{1-\theta}}\textnormal{diam}(I_j)^{\frac{1}{1-\theta}}}K(y^{-1}z)\phi_j(y)dy\right|dz\\
    &=\smallint\limits_{ |z|>\textnormal{diam}(I_j)  }\left|\, \smallint\limits_{|y|<2^{-\frac{1}{1-\theta}}\textnormal{diam}(I_j)^{\frac{1}{1-\theta}}}(K(z)-K(y^{-1}z))\phi_j(y)dy\right|dz\\
     &\leq   \smallint\limits_{|y|<2^{-\frac{1}{1-\theta}}\textnormal{diam}(I_j)^{\frac{1}{1-\theta}}} \smallint\limits_{ |z|>\textnormal{diam}(I_j)  }|K(y^{-1}z)-K(z)|dz|\phi_j(y)|dy\\
     &\lesssim \frac{1}{R_j^Q}  \smallint\limits_{|y|<R_j} \smallint\limits_{ |z|> 2R_j^{(1-\theta)}}|K(y^{-1}z)-K(z)|dz\,dy\\
     &\leq [K]_{H_{\infty,\theta}(G)},
\end{align*}where we have used that $2R_j^{1-\theta}=\textnormal{diam}(I_j).$ 
Now, the inequalities above allow us to finish the estimate of $\Vert T(b-\tilde{b})\Vert_{L^1(G\setminus I^*)}.$ Indeed,
\begin{align*}
    \Vert T(b-\tilde{b})\Vert_{L^1(G\setminus I^*)} &\leq \sum_{j} \smallint\limits_{I_j} \smallint\limits_{ G\setminus I^*}| K(y^{-1}x)-\phi_j\ast K(y^{-1}x)|dx |b_j(y)|dy \\
    &\leq\sum_{j}\smallint\limits_{I_j} \smallint\limits_{ |z|>\textnormal{diam}(I_j)} |K(z)-\phi_j\ast K(z)|dz |b_j(y)|dy\\
    &\lesssim [K]_{H_{\infty,\theta}(G)} \sum_{j}\smallint\limits_{I_j}|b_j(y)|dy\leq [K]_{H_{\infty,\theta}(G)}\Vert b \Vert_{L^1}\\
    &\lesssim [K]_{H_{\infty,\theta}(G)}\Vert f\Vert_{L^1}.
\end{align*}Putting together the estimates above we deduce that
\begin{align*}
    & |\{x\in G\setminus I^*:|Tb(x)|>\frac{\alpha}{2}\}|\leq   \frac{4}{\alpha}\Vert Tb-T\tilde{b}\Vert_{L^1(G\setminus I^*)}+|\{x\in G\setminus I^*:|T\tilde b(x)|>\frac{\alpha}{4}\}|\\
     &\lesssim \frac{4}{\alpha}[K]_{H_{\infty,\theta}(G)}\Vert f\Vert_{L^1}+|\{x\in G\setminus I^*:|T\tilde b(x)|>\frac{\alpha}{4}\}|.
\end{align*}Now, we will estimate the second term on the right hand side of this inequality. First, note that
$$
    |\{x\in G\setminus I^*:|T\tilde b(x)|>\frac{\alpha}{4}\}|\leq   |\{x\in G\setminus I^*:|T\tilde b(x)|^2>\frac{\alpha^2}{16}\}|
$$
$$\leq \frac{16}{\alpha^2}\Vert T\tilde{b}\Vert_{L^2}^2. 
$$ 

Now, using  \eqref{LxiG} we deduce that $T(1+\mathcal{R})^{\frac{Q\theta}{2\nu}}$ is bounded on $L^2$ (see Remark \ref{Rem:Four:Conf}), and then
\begin{equation}\label{asterisque}\Vert T\tilde{b}\Vert^2_{L^2}\leq \Vert T(1+\mathcal{R})^{\frac{Q\theta}{2\nu}}\Vert_{\mathscr{B}(L^2)}^2\Vert(1+\mathcal{R})^{-\frac{Q\theta}{2\nu}}\tilde{b}\Vert_{L^2}^2.
\end{equation}
In the case of $G=\mathbb{R}^n,$ and of the positive Laplace operator $\mathcal{R}=-\Delta,$ it was proved by Fefferman that the function $F=(1+\Delta)^{-\frac{Q\theta}{2\nu}}\tilde{b},$ admits a nice decomposition $F=F_1+F_2,$ where  $  \Vert F_2\Vert_{L^2}^2\leq C\alpha\gamma\Vert f\Vert_{L^1},$ and $F_{1}=\sum_{j:\textnormal{diam}(I_j)<1}F_{1}^{j},$ where $\Vert F_{1}^{j}\Vert_{L^2}^2\leq A'\alpha^2|I_j|.$ We will extend Fefferman's decomposition to an arbitrary Rockland operator $\mathcal{R}$ in Lemma \ref{lemma:Fefferman} below.

\subsection{Part 4: A Fefferman type decomposition for Rockland operators}
In order to estimate the $L^2$-norm $\Vert(1+\mathcal{R})^{-\frac{Q\theta}{2\nu}}\tilde{b}\Vert_{L^2}^2$, let us use the following lemma whose proof we postpone for a moment.
\begin{lemma}\label{lemma:Fefferman} The function $F:=(1+\mathcal{R})^{-\frac{Q\theta}{2\nu}}\tilde{b}$ can be decomposed in a sum $F=F_1+F_2,$ where $  \Vert F_2\Vert_{L^2}^2\leq C\alpha\gamma\Vert f\Vert_{L^1},$ and $F_1$ is also a sum of  functions $F_{1}^{j}$ with the following property:
\begin{itemize}
    \item There exists $M_0\in \mathbb{N},$ and $A'>0,$ such that $F_{1}=\sum_{j:\textnormal{diam}(I_j)<1}F_{1}^{j}$,  $\Vert F_{1}^{j}\Vert_{L^2}^2\leq A'\alpha^2|I_j|,$ and for any $x\in G,$ there are at most $M_0$ values of $j$ such that  $F_1^j(x)\neq 0.$ 
\end{itemize}

\end{lemma}

Let us continue with the proof of Theorem \ref{main:th}. 
Using  Lemma \ref{lemma:Fefferman} and the inequalities in \eqref{I:j} and \eqref{asterisque} we have that
\begin{align*}
\Vert T\tilde{b}\Vert_{L^2}^2&\lesssim \Vert F_1\Vert_{L^2}^2+\Vert F_2\Vert_{L^2}^2 \lesssim \alpha\gamma\Vert f\Vert_{L^1}+ \sum_{j}\Vert F_1^{j} \Vert_{L^2}^2\\
&\lesssim \alpha\gamma\Vert f\Vert_{L^1}+ \sum_j\alpha^2|I_j|\lesssim \alpha\gamma\Vert f\Vert_{L^1}+\alpha^2(\gamma\alpha)^{-1}\Vert f\Vert_{L^1}\\
    &= (\gamma^{-1}+\gamma)\alpha\Vert f\Vert_{L^1}.
\end{align*}
Consequently
$$
    |\{x\in G\setminus I^*:|T\tilde b(x)|>\frac{\alpha}{4}\}|\lesssim \frac{16}{\alpha^2}\Vert T\tilde{b}\Vert_{L^2}^2 
\lesssim (\gamma^{-1}+\gamma)\alpha^{-1}\Vert f\Vert_{L^1}.
$$ Thus, we have proved that
\begin{equation}
   |\{x\in \mathbb{R}^n:|Tf(x)|>\alpha\}|\leq C_{\gamma,[K]_{H_{\infty,\theta}(G)}}\alpha^{-1}\Vert f\Vert_{L^1},
\end{equation}with $C:=C_{\gamma,[K]_{H_{\infty,\theta}(G)}}$ independent of $f.$ 

Because $\gamma$ is fixed, we have proved the weak $(1,1)$ type of $T$. Thus, the proof is complete once we prove the statement in Lemma \ref{lemma:Fefferman}. To do so, let us consider the right-convolution kernel $k_\theta := (1+\mathcal{R})^{-\frac{Q\theta}{2\nu}}\delta$ of the operator $(1+\mathcal{R})^{-\frac{Q\theta}{2\nu}}$, and let us split the function $(1+\mathcal{R})^{-\frac{Q\theta}{2\nu}}\tilde{b}(x)$ as follows:

\begin{align*}
   (1+\mathcal{R})^{-\frac{Q\theta}{2\nu}}\tilde{b}(x)&=\sum_{j} (1+\mathcal{R})^{-\frac{Q\theta}{2\nu}}\tilde{b}_{j}(x)=\sum_{j} \tilde{b}_{j}\ast k_\theta(x)\\
   &=\sum_{j:x\sim I_j } \tilde{b}_{j}\ast k_\theta(x)+\sum_{j:x\nsim I_j} \tilde{b}_{j}\ast k_\theta(x)=:G_{1}(x)+G_{2}(x),
\end{align*}where $G_{1}(x):=\sum_{j:x\sim I_j} \tilde{b}_{j}\ast k_\theta(x)$ and $G_{2}(x):=\sum_{j:x\nsim I_j} \tilde{b}_{j}\ast k_\theta(x).$ We denote by \( x \sim I_j \) if \( x \) belongs to \( I_j \) or to some \( I_{j'} \) with a non-empty intersection with \( I_j \). By the properties of these sets, there are at most \( M_0 \) sets \( I_{j'} \) such that \( I_j \cap I_{j'} \neq \emptyset \). Also, the notation \( x \nsim I_j \) will be used to define the opposite of the previous property.

Let us prove the estimate
\begin{equation}\label{esti:G2}
    \Vert G_2\Vert_{L^2}^2\leq C\alpha\gamma\Vert f\Vert_{L^1}.
\end{equation} Observe that
\begin{align*}
   \Vert G_2\Vert_{L^1} &=\smallint\limits_G| \sum_{j:x\nsim I_j} {b}_{j}\ast \phi_j\ast k_\theta(x)|dx \leq \sum_{j:x\nsim I_j}\smallint\limits_G| {b}_{j}\ast \phi_j\ast k_\theta(x)|dx \\
   &\leq  \sum_{j}\| {b}_{j}\ast \phi_j\ast k_\theta\|_{L^1}\leq  \sum_{j}\| {b}_{j}\Vert_{L^1}\Vert \phi_j\ast k_\theta\|_{L^1}. 
\end{align*}Note that $(1+\mathcal{R})^{-\frac{Q\theta}{2\nu}}$ is a pseudo-differential operator of order $-Q\theta/2,$ and consequently, its kernel satisfies the estimate (see \cite[Theorem 5.4.1]{FischerRuzhanskyBook})
\begin{equation*}
    |k_{\theta}(x)|\leq C_N|x|^{{-(-\frac{Q\theta}{2}+Q)}}\lesssim |x|^{-{Q(1-\frac{\theta}{2})}},\,x\in B(e,N)\setminus\{e\},
\end{equation*}
for any \( N \in \mathbb{N} \), while outside of any ball \( B(e, N) \), one has, for any \( M > 0 \), that there exists \( C_M > 0 \) such that (see Theorem 5.4.1 of \cite{FischerRuzhanskyBook})
\[
|k_{\theta}(x)| \leq C_M(1 + |x|)^{-M}.
\]
The condition \( 0 < \theta < 1 \) implies that \( k_\theta \) is an integrable distribution, and then
\begin{equation*}
  \Vert \phi_j\ast k_\theta\|_{L^1} \lesssim \Vert \phi_j\|_{L^1} \| k_\theta\|_{L^1}=\| k_\theta\|_{L^1}<\infty.   
\end{equation*}
So, we have that
\begin{align*}
     \Vert G_2\Vert_{L^1}\lesssim \sum_{j}\|b_j\|_{L^1}\lesssim\Vert f \Vert_{L^1}.
\end{align*}So, for the proof of \eqref{esti:G2},  in view of the inequality $ \Vert G_2\Vert^2_{L^2}\leq \Vert G_2\Vert_{L^1} \Vert G_2\Vert_{L^\infty}$ is enough to show that
$$ \Vert G_2\Vert_{L^\infty}\lesssim\alpha\gamma.$$
To do this, let us consider $j$ such that  $x\nsim I_j.$ Since $\tilde{b}_{j}\ast k_\theta={b}_{j}\ast \phi_j\ast k_\theta,$ one has that
\begin{align*}
 | {b}_{j}\ast \phi_j\ast k_\theta(x)|&\leq \smallint\limits_{I_j}|(\phi_j\ast k_\theta)(y^{-1}x)||b_j(y)|dy\leq \sup_{y\in I_j}|(\phi_j\ast k_\theta)(y^{-1}x)|  \smallint\limits_{I_j}|b_j(y)|dy\\
 &=\sup_{y\in I_j}|\phi_j\ast k_\theta(y^{-1}x)| |I_j|\times \frac{1}{|I_j|} \smallint\limits_{I_j}|b_j(y)|dy.
\end{align*}
To continue, we follow the observation of Fefferman in \cite[Page 26]{Fefferman1970}, which states that, in view of the property \( x \nsim I_j \), we have that \( \phi_j \ast |k_\theta|(y^{-1}x) \) is essentially constant over the ball \( I_j = B(x_j, r_j) \), and we can estimate
\begin{equation}
    \sup_{y\in I_j}|\phi_j\ast k_\theta(y^{-1}x)| |I_j|\leq \sup_{y\in I_j}\phi_j\ast |k_\theta|(y^{-1}x)| |I_j|\lesssim \smallint\limits_{I_j}\phi_j\ast |k_\theta(y'^{-1}x)|dy'.
\end{equation}On the other hand, using again the positivity of \( \phi_j \) leads to
\begin{align*}
    \smallint\limits_{I_j}|\phi_j\ast k_\theta(y'^{-1}x)|dy' \frac{1}{|I_j|} \smallint\limits_{I_j}|b_j(y)|dy
    &\leq\smallint\limits_{G}(\phi_j\ast |k_\theta|)(y'^{-1}x) \left( \frac{1}{|I_j|} \smallint\limits_{I_j}|b_j(y)|dy\right)1_{I_{j}}(y')dy'\\
    &= \left(\frac{1}{|I_j|} \smallint\limits_{I_j}|b_j(y)|dy\times1_{I_j}\right)\ast\phi_j\ast| k_\theta|(x).
\end{align*}Consequently, we have that
\begin{align*}
    |G_2(x)|&\leq \sum_{j:x\nsim I_j} |\tilde{b}_{j}\ast k_\theta(x)|\lesssim  \sum_{j:x\nsim I_j}\left(\frac{1}{|I_j|} \smallint\limits_{I_j}|b_j(y)|dy\times1_{I_j}\right)\ast\phi_j\ast |k_\theta|(x)\\
    &\lesssim\sum_{j:x\nsim I_j}\gamma\alpha\times1_{I_j}\ast\phi_j\ast |k_\theta|(x)=\smallint\limits_{G}\sum_{j:x\nsim I_j}\gamma\alpha\times1_{I_j}\ast\phi_j(z) |k_\theta(z^{-1}x)|dz\\
    &\lesssim \gamma\alpha \Vert k_\theta\Vert_{L^1}\left\Vert \sum_{j:x\nsim I_j}1_{I_j}\ast\phi_j \right\Vert_{L^{\infty}}.
\end{align*} By observing that  the supports of the functions $1_{I_j}\ast\phi_j $ have bounded overlaps we have that $$\left\Vert \sum_{j:x\nsim I_j}1_{I_j}\ast\phi_j \right\Vert_{L^{\infty}}<\infty,$$ and that $\Vert G_2\Vert_{L^\infty}\lesssim\gamma\alpha.$ It remains only to prove that $\Vert G_1\Vert_{L^2}^2\lesssim \alpha \gamma\Vert f\Vert_{L^1}.$ Let us define
\begin{equation}\label{F:j:iproof}G^{j}(x):= \begin{cases}{b}_{j}\ast \phi_j\ast k_\theta(x),& \text{ }x\in I_j,
\\0 ,& \,\text{otherwise. } \end{cases}
\end{equation}Then \( G_1 = \sum_{j} G^{j} \), and in view of the finite overlap of the balls \( I_{j} \)'s, there exists \( M_0 \in \mathbb{N} \) such that for any \( x \in G \), \( G^{j}(x) \neq 0 \) for at most \( M_0 \) values of \( j \). Therefore, we have that
\begin{align*}
    \smallint_G|G_1(x)|^2dx &\leq M_0 \sum_{j}\smallint_G|G^{j}(x)|^2dx =\sum_{j}\smallint_{I_j}|{b}_{j}\ast \phi_j\ast k_\theta(x)|^2dx\\
    &\leq \sum_{j}\Vert {b}_{j}\Vert_{L^1}^2\Vert \phi_j\ast k_\theta\Vert_{L^2}^2\leq \sum_{j}\alpha^2\gamma^2|I_j|^2\Vert \phi_j\ast k_\theta\Vert_{L^2}^2.
\end{align*} 
To estimate the $L^2$-norm $\Vert \phi_j\ast k_\theta\Vert_{L^2}^2$, we will use the dilation properties of the spectral calculus of the Rockland operator $\mathcal{R}$.
\subsection{Final Part: Control of the Bessel operator associated to the Rockland operator}

Now, let us apply the Plancherel theorem:
$$ \Vert \phi_j\ast k_\theta\Vert_{L^2(G)}^2=\Vert(1+\pi(\mathcal{R}))^{-\frac{Q\theta}{2\nu}}\widehat{\phi}_j\Vert^2_{L^2(\widehat{G})}. $$
Define for any $j,$ 
\begin{equation}\label{defi:psij:graded}
\psi_{j}(x):=R_j^{Q}\phi_{j}(R_j\cdot x),\,x\in G.    
\end{equation}
Observe that for any $j,$
\begin{align}\label{psijisphi}
\psi_{j}(x):=R_j^{Q}\phi_{j}(R_j\cdot x)=   R_j^{Q}R_j^{-Q}\phi(R_j\cdot R_j^{-1}\cdot x)=\phi(x). 
\end{align}However, in order to use the dilation properties of the Fourier transform of distributions, let us keep in mind the identity for \( \psi_j = \phi \) in \eqref{defi:psij:graded}. Also, let us remark that we have the following identities for the function \( \phi_j \):
\begin{equation*}
    \phi_j(x)=R_j^{-Q}\psi_j\left({R_j}^{-1}\cdot x\right),\quad\,x\in G.
\end{equation*} as well as for its Fourier transform (in view of \eqref{Eq:dilatedFourier})
\begin{align*}
    \widehat{\phi_j}( \pi)= \widehat{\psi_j}(R_j\cdot \pi),\,\pi\in \widehat{G}.
\end{align*}
Also, note that $\Vert \psi_j\Vert_{L^\infty}=1,$ and $\textnormal{supp}(\psi_j)\subset B(e,1).$  Using the Plancherel theorem, we have that
\begin{align*}
  \Vert \phi_j\ast k_\theta\Vert_{L^2(G)}^2&=\Vert(1+\pi(\mathcal{R}))^{-\frac{Q\theta}{2\nu}}\widehat{\phi}_j\Vert^2_{L^2(\widehat{G})} \\
  &=  \Vert (1+\pi(\mathcal{R}))^{-\frac{Q\theta}{2\nu}}\widehat{\psi_j}(R_j\cdot \pi)\Vert^2_{L^2(\widehat{G})} \\
  &= \smallint\limits_{\widehat{G}} \Vert(1+\pi(\mathcal{R}))^{-\frac{Q\theta}{2\nu}}\widehat{\psi_j}(R_j\cdot \pi)\Vert_{\textnormal{HS}}^2d\pi \\
  &={R_j}^{-Q} \smallint\limits_{\widehat{G}} \Vert(1+({R_j}^{-1}\cdot\pi)(\mathcal{R}))^{-\frac{Q\theta}{2\nu}}\widehat{\psi_j}( \pi)\Vert_{\textnormal{HS}}^2d\pi,
\end{align*}where we have used the rescaling property in Remark \ref{DilationLebesgue}. Let us write:
\begin{align*}
    ({R_j}^{-1}\cdot\pi)(\mathcal{R})=\widehat{\mathcal{R}\delta}({R_j}^{-1}\cdot\pi).
\end{align*}
Using  the functional calculus  for $\varkappa(t):=(1+t)^{-\frac{Q\theta}{2\nu}},$ we have that
\begin{align*}
\forall r>0,\,\forall x\in G,\,\,    \varkappa(r^{\nu}\mathcal{R})\delta(x)=r^{-Q}\varkappa(\mathcal{R})\delta(r^{-1}\cdot x),
\end{align*}that is equivalent to say that
\begin{align*}
\forall r>0,\,\forall x\in G,\,\,    (1+r^{\nu}\mathcal{R})^{-\frac{Q\theta}{2\nu}}\delta(x)=r^{-Q}(1+\mathcal{R})^{-\frac{Q\theta}{2\nu}}\delta(r^{-1}\cdot x).
\end{align*}Taking the Fourier transform in both sides of this equality, and using the functional calculus of Rockland operators, we obtain
\begin{align*}
   \forall r>0,\,\forall \pi\in \widehat{G},\,\,    (1+r^{\nu}\pi(\mathcal{R}))^{-\frac{Q\theta}{2\nu}}&=\mathscr{F}[r^{-Q}(1+\mathcal{R})^{-\frac{Q\theta}{2\nu}}\delta(r^{-1}\cdot )](\pi) \\
   &=\mathscr{F}[(1+\mathcal{R})^{-\frac{Q\theta}{2\nu}}\delta(\cdot )](r\cdot \pi).
\end{align*}
Thus, we have
\begin{equation}\label{power:identity}
   \forall r>0,\,\forall \pi\in \widehat{G},\,\,    (1+r^{\nu}\pi(\mathcal{R}))^{-\frac{Q\theta}{2\nu}}=\mathscr{F}[(1+\mathcal{R})^{-\frac{Q\theta}{2\nu}}\delta(\cdot )](r\cdot \pi).  
\end{equation}
By taking $r=R_j^{-1}$ in \eqref{power:identity} we have that
\begin{align*}
    \forall j,\,\forall \pi\in \widehat{G},\,\,    (1+R_j^{-\nu}\pi(\mathcal{R}))^{-\frac{Q\theta}{2\nu}}=\mathscr{F}[(1+\mathcal{R})^{-\frac{Q\theta}{2\nu}}\delta(\cdot )](R_{j}^{-1}\cdot \pi)=(1+(R_{j}^{-1}\cdot \pi)(\mathcal{R}))^{-\frac{Q\theta}{2\nu}}.
\end{align*}

Note that
\begin{align*}
 {R_j}^{-Q} \smallint\limits_{\widehat{G}} \Vert(1+({R_j}^{-1}\cdot\pi)(\mathcal{R}))^{-\frac{Q\theta}{2\nu}}\widehat{\psi_j}( \pi)\Vert_{\textnormal{HS}}^2d\pi&={R_j}^{-Q} \smallint\limits_{\widehat{G}} \Vert(1+R_j^{-\nu}\pi(\mathcal{R}))^{-\frac{Q\theta}{2\nu}}\widehat{\psi_j}( \pi)\Vert_{\textnormal{HS}}^2d\pi\\
 &={R_j}^{-Q}R_{j}^{Q\theta} \smallint\limits_{\widehat{G}} \Vert(R_{j}^{\nu}+\pi(\mathcal{R}))^{-\frac{Q\theta}{2\nu}}\widehat{\psi_j}( \pi)\Vert_{\textnormal{HS}}^2d\pi.
\end{align*}
Now, let us prove that for $a=R_j^{\nu}$ and $b=\frac{Q\theta}{2\nu},$ and $\tau(\pi):=\widehat{\psi_j}( \pi),$ we have the inequality:
\begin{equation}
     \Vert \pi(a+\mathcal{R})^{-b}\tau(\pi)\Vert_{\textnormal{HS}}^2\leq\Vert \pi(\mathcal{R})^{-b}\tau(\pi)\Vert_{\textnormal{HS}}^2.
\end{equation} Indeed, let us denote by  $\{dE_{\pi(\mathcal{R})}(\lambda)\}_{\lambda>0}$ the spectral measure  associated to the operator $\pi(\mathcal{R}).$ If $B_{\pi}=\{e_{\pi,k}\}_{k=1}^\infty$ is a basis of the representation space $H_\pi,$ then,
\begin{align*}
    \Vert \pi(a+\mathcal{R})^{-b}\tau(\pi)\Vert_{\textnormal{HS}}^2&=\sum_{k=1}^\infty\left\Vert  \pi(a+\mathcal{R})^{-b} \tau(\pi)e_{\pi,k} \right\Vert_{H_\pi}^2\\
    &= \sum_{k=1}^\infty\left\Vert \smallint\limits_{0}^\infty(a+\lambda)^{-b}dE_{\pi(\mathcal{R})}(\lambda)\tau(\pi)e_{\pi,k} \right\Vert_{H_\pi}^2\\
    &=\sum_{k=1}^\infty \smallint\limits_{0}^\infty(a+\lambda)^{-2b}d\Vert E_{\pi(\mathcal{R})}(\lambda)\tau(\pi)e_{\pi,k} \Vert_{H_\pi}^2\\
    &\leq\sum_{k=1}^\infty \smallint\limits_{0}^\infty\lambda^{-2b}d\Vert E_{\pi(\mathcal{R})}(\lambda)\tau(\pi)e_{\pi,k} \Vert_{H_\pi}^2\\
    &=\sum_{k=1}^\infty\left\Vert \smallint\limits_{0}^\infty\lambda^{-b}dE_{\pi(\mathcal{R})}(\lambda)\tau(\pi)e_{\pi,k} \right\Vert_{H_\pi}^2\\
    &=\sum_{k=1}^\infty\left\Vert  \pi(\mathcal{R})^{-b} \tau(\pi)e_{\pi,k} \right\Vert_{H_\pi}^2\\
    &=\Vert \pi(\mathcal{R})^{-b}\tau(\pi)\Vert_{\textnormal{HS}}^2,
\end{align*}as desired.
So, we can estimate
\begin{align*}
 {R_j}^{-Q} \smallint\limits_{\widehat{G}} \Vert(1+({R_j}^{-1}\cdot\pi)(\mathcal{R}))^{-\frac{Q\theta}{2\nu}}\widehat{\psi_j}( \pi)\Vert_{\textnormal{HS}}^2d\pi &\lesssim{R_j}^{-Q}R_{j}^{Q\theta} \smallint\limits_{\widehat{G}} \Vert\pi(\mathcal{R})^{-\frac{Q\theta}{2\nu}}\widehat{\psi_j}( \pi)\Vert_{\textnormal{HS}}^2d\pi\\
 &={R_j}^{-Q(1-\theta)} \smallint\limits_{\widehat{G}} \Vert\pi(\mathcal{R})^{-\frac{Q\theta}{2\nu}}\widehat{\psi_j}( \pi)\Vert_{\textnormal{HS}}^2d\pi\\
 &={R_j}^{-Q(1-\theta)}\Vert\mathcal{R}^{-\frac{Q\theta}{2\nu}}\psi_j\Vert^2_{L^2(G)}\\
 &={R_j}^{-Q(1-\theta)}\Vert\mathcal{R}^{-\frac{Q\theta}{2\nu}}\phi\Vert^2_{L^2(G)}\\
 &\lesssim|I_j|^{-1},
\end{align*}where we have used \eqref{psijisphi} and the estimate $|I_j|\sim R_{j}^{Q(1-\theta)}$ because $2R_j^{1-\theta}=\textnormal{diam}(I_j)$. Thus, we deduce that
\begin{align*}
    \smallint_G|G_1(x)|^2dx &\leq \sum_{j}\alpha^2\gamma^2|I_j|^2\Vert \phi_j\ast k_\theta\Vert_{L^2}^2\lesssim \sum_{j} \alpha^2\gamma^2|I_j|^2|I_j|^{-1}=\alpha^2\gamma^2\sum_{j}|I_j|\\
    &\lesssim \alpha\gamma\Vert f\Vert_{L^1}\lesssim_\gamma \alpha\Vert f\Vert_{L^1},
\end{align*} in view of \eqref{I:j}. Consequently, we can take $F_2:=G_2,$ $F_1:=G_1$ and  $F_{1}^{j}:=G_{1}^{j}.$ Thus, the proof of  Lemma \ref{lemma:Fefferman} is complete as well as the proof of Theorem  \ref{main:th}.
\end{proof}

\noindent {\bf{Conflict of interests statement.}} On behalf of all authors, the corresponding author states that there is no conflict of interest.\\

\noindent  {\bf{Declaration.}} We confirm that the order of authors listed in the manuscript has been approved by all named authors. We confirm that the manuscript has been read and approved by all named authors.\\

\noindent {\bf{Acknowledgements.}}
We would like to thank the referee for his/her valuable comments, which significantly helped to improve this article. In particular, we appreciate his/her suggestion to include an additional hypothesis in our theorem. The authors also thank Cody Stockdale for his comments on an earlier version of this work.

\bibliographystyle{amsplain}

\end{document}